\documentclass[10pt]{article}
\usepackage{amsfonts,amssymb,amsmath,amsthm,epsfig,euscript,verbatim,graphicx}
\usepackage{textcomp}
\usepackage{enumerate}
\usepackage{epsfig}
\usepackage{tikz}
\usepackage{subcaption}
\usepackage[autostyle]{csquotes}
\usepackage{lmodern}
\usepackage[T1]{fontenc}

\usepackage{palatino}
\usepackage{mathpazo}

\begingroup
    \makeatletter
    \@for\theoremstyle:=definition,remark,plain\do{%
        \expandafter\g@addto@macro\csname th@\theoremstyle\endcsname{%
            \addtolength\thm@preskip\parskip
            }%
        }
\endgroup

\newtheorem{theorem}{Theorem}
\newtheorem{lemma}[theorem]{Lemma}

\newtheorem{conjecture}{Conjecture}

\makeatletter
\newfont{\footsc}{cmcsc10 at 8truept}
\newfont{\footbf}{cmbx10 at 8truept}
\newfont{\footrm}{cmr10 at 10truept}
\renewenvironment{abstract}{
	\begin{list}{}%
	{\setlength{\rightmargin}{1in}%
	\setlength{\leftmargin}{1in}}%
	\item[]\ignorespaces\begin{small}}%
	{\end{small}\unskip\end{list}%
}

\newcommand{\sds}[2]{%
	\begin{tikzpicture}[baseline=(current bounding box.south),scale={#2}]
		\foreach[count=\x] \i in {#1}{
			\foreach \j in {1,...,\i}{
				\draw[gray!85!blue, thick, fill=gray!20] (\x, \j) rectangle (\x+1,\j+1);
			}
			\node at (\x+.5,.6) {\scriptsize $\i$};
		}
	\end{tikzpicture}%
}
\newcommand{\sdsplus}[3]{%
	\begin{tikzpicture}[baseline=(current bounding box.south),scale={#2}]
		\foreach[count=\x] \i in {#1}{
			\foreach \j in {1,...,\i}{
				\draw[gray!85!blue, thick, fill=gray!20] (\x, \j) rectangle (\x+1,\j+1);
			}
			\node at (\x+.5,.6) {\scriptsize $\i$};
		}
		{#3}
	\end{tikzpicture}%
}

\newcommand{\clusterOne}{%
\begin{tikzpicture}[every node/.style={inner sep=0,outer sep=0},baseline=-0.75ex]
	\node at (0,0) {524433222};	
	\draw[thin] (-0.8, 0.16) -- (-0.12, 0.16);
	\draw[thin] (-0.44, 0.20) -- (0.24, 0.20);
	\draw[thin] (.1, 0.24) -- (0.84, 0.24);
\end{tikzpicture}%
}

\newcommand{\clusterTwo}{%
\begin{tikzpicture}[every node/.style={inner sep=0,outer sep=0},baseline=-0.75ex]
	\node at (0,0) {524433222};	
	\draw[thin] (-0.8, 0.16) -- (-0.12, 0.16);
	\draw[thin] (-0.44, 0.20) -- (0.24, 0.20);
	\draw[thin] (-.28, 0.24) -- (0.46, 0.24);
	\draw[thin] (.1, 0.28) -- (0.84, 0.28);
\end{tikzpicture}%
}

\newcommand{\clusterThree}{%
\begin{tikzpicture}[every node/.style={inner sep=0,outer sep=0},baseline=-0.75ex]
	\node at (0,0) {313223122};	
	\draw[thin] (-0.8, 0.16) -- (-0.12, 0.16);
	\draw[thin] (-0.44, 0.20) -- (0.24, 0.20);
	\draw[thin] (.1, 0.24) -- (0.84, 0.24);
\end{tikzpicture}%
}

\newcommand*\patchAmsMathEnvironmentForLineno[1]{%
  \expandafter\let\csname old#1\expandafter\endcsname\csname #1\endcsname
  \expandafter\let\csname oldend#1\expandafter\endcsname\csname end#1\endcsname
  \renewenvironment{#1}%
     {\linenomath\csname old#1\endcsname}%
     {\csname oldend#1\endcsname\endlinenomath}}%
\newcommand*\patchBothAmsMathEnvironmentsForLineno[1]{%
  \patchAmsMathEnvironmentForLineno{#1}%
  \patchAmsMathEnvironmentForLineno{#1*}}%
\AtBeginDocument{%
\patchBothAmsMathEnvironmentsForLineno{equation}%
\patchBothAmsMathEnvironmentsForLineno{align}%
\patchBothAmsMathEnvironmentsForLineno{flalign}%
\patchBothAmsMathEnvironmentsForLineno{alignat}%
\patchBothAmsMathEnvironmentsForLineno{gather}%
\patchBothAmsMathEnvironmentsForLineno{multline}%
}

\usepackage{mathtools}
\usepackage{parskip}
\usepackage{lineno}
\usepackage[small]{titlesec}

\usepackage[square,comma,numbers,sort&compress]{natbib}

\renewcommand{\P}{\mathbb{P}}

\newpagestyle{main}[\small]{
	\headrule
	\sethead[\usepage][][]
	{\sc Shift Equivalence in the Generalized Factor Order}{}{\usepage}
}

\title{\sc Shift Equivalence in the\\Generalized Factor Order}
\author{
	Jennifer Fidler\\ \small The Cottesloe School\\ \small Wing, UK
	\and
	Daniel Glasscock\\ \small Department of Mathematics\\ \small Ohio State University\\ \small Columbus, OH
	\and	
	Brian Miceli\\ \small Department of Mathematics\\ \small Trinity University\\ \small San Antonio, TX
	\and
	Jay Pantone\\ \small Department of Mathematics\\ \small Dartmouth College\\ \small Hanover, NH
	\and
	Min Xu\\ \small School of Computer Science\\ \small Carnegie Mellon University\\ \small Pittsburgh, PA
}

\titleformat{\section}{\large\sc}{\thesection.}{1em}{}
\date{}

\begin{document}
\maketitle

\begin{abstract}
We provide a geometric condition that guarantees strong Wilf equivalence in the generalized factor order. This provides a powerful tool for proving specific and general Wilf equivalence results, and several such examples are given.
\end{abstract}

\section{Introduction}
\label{sec:intro}

We say that a nonempty word $u = u_1u_2\ldots u_k$ is a \emph{factor} of a word $w$ if the letters of $u$ appear consecutively in $w$, i.e., if $w = w^{(1)}uw^{(2)}$ for some words $w^{(1)}$ and $w^{(2)}$. More generally, given a poset $P$ and words $u$ and $w$ whose letters are from $P$, we say that $u$ is a \emph{generalized factor} of $w$ if $w = w^{(1)}vw^{(2)}$ for a word $v$ of the same length of $u$ with the property that $u_i \leq_P v_i$ for all $i$. Each such $v$ is called an \emph{embedding of} $u$ \emph{in} $w$, and if no such embeddings exist, then we say that $w$ \emph{avoids} $u$. The poset induced by the generalized factor relation is called the \emph{generalized factor order over $P$}. 

This paper is concerned only with the case where $P=\P$, the positive integers with the usual order. For a word $w \in \P^*$, we define $|w|$ to be the length of $w$ and $\|w\|$ to be the sum of the letters of $w$. Kitaev, Liese, Remmel, and Sagan~\cite{KLRS} introduced the generalized factor order over $\P$ and defined two words $u$ and $v$ to be \emph{Wilf equivalent} if the number of words of length $n$ with sum $m$ that avoid $u$ is the same as the number of words of length $n$ and sum $m$ that avoid $v$, for all $n$ and $m$. Defining the generating function
\[
	F_u(x,y) = \sum_{\mathclap{\substack{w \,:\, u\text{ not a}\\\text{gen. fact. of $w$}}}} \;x^{|w|}y^{\|w\|},
\]
we see that $u$ and $v$ are Wilf equivalent if and only if $F_u(x,y) = F_v(x,y)$. 


%

The notion of Wilf equivalence can be refined. Define $\eta_u(w)$ to be the number of embeddings of $u$ in $w$. For example, if $u = 154$ then $\eta_u(16563) = 2$, because both $165$ and $656$ are embeddings of $u$ in $w$. Define
\[
	A_u(x,y,z) = \sum_{\mathclap{w \in \P^*}}x^{|w|}y^{\|w\|}z^{\eta_u(w)}.
\]
Following Pantone and Vatter~\cite{PV}, we say that $u$ and $v$ are \emph{strongly Wilf equivalent} if
\[
	A_u(x,y,z) = A_v(x,y,z).
\]
Note that $F_u(x,y) = A_u(x,y,0)$, and so $u$ and $v$ are Wilf equivalent if and only if $A_u(x,y,0) = A_v(x,y,0)$. 

Generalizations of this notion of embedding can be found in the works of Chamberlain, Cochran, Ginsburg, Miceli, Riehl, and Zhang~\cite{MR} and Langley, Liese, and Remmel \cite{LLR, LLR*}, while Hadjiloucas, Michos, and Savvidou~\cite{HMS} investigate the similar notion of super-strong Wilf equivalence. Various results about the Wilf equivalence of families of words are known (see~\cite{MR, KLRS, LLR, LLR*}), yet a tantalizing conjecture from~\cite{KLRS} remains open.

\begin{conjecture}[Rearrangment Conjecture,~\cite{KLRS}] If $u$ and $v$ are Wilf equivalent, then $u$ and $v$ are rearrangements of each other. That is, they have the same multiset of letters. \label{conj:rearrange}
\end{conjecture}

The converse of this statement is false; it is straightforward to compute $A_u(x,y,z)$ for any particular $u$ and so it can be verified that $A_{132}(x,y,0) \neq A_{312}(x,y,0)$. 

Recent work of Pantone and Vatter~\cite{PV} uses the Cluster Method to prove a weakening of the Rearrangement Conjecture. They show that if $u$ and $v$ are strongly Wilf equivalent, then they must be rearrangements of each. Moreover, computational evidence led them to conjecture that, surprisingly, Wilf equivalence and strong Wilf equivalence are the same condition. That is, they conjecture that $u$ and $v$ are Wilf equivalent if and only if they are strongly Wilf equivalent.


Herein, we prove a geometric result that partially classifies when two words are strongly Wilf equivalent. With this tool in hand, we are able to prove in a unified way conjectures from~\cite{KLRS} and reprove several theorems from~\cite{KLRS} and~\cite{PV}.


\section{Cluster Method}
\label{section:cluster}

The Cluster Method of Goulden and Jackson~\cite{GJ} is a framework that can be used to study consecutive pattern avoidance for many types of combinatorial objects. We will present a brief outline of its application to the generalized factor order; we refer the reader to~\cite{PV} for a more detailed explanation in this context, and to~\cite{EN} for an application to permutation pattern avoidance.

Given a nonempty word $u$ over $\P$, an \emph{$m$-cluster of $u$} is a word $c$ together with $m$ marked occurrences of $u$ in $c$ such that every letter of $c$ is part of some marked occurrence and each consecutive pair of occurrences overlaps in at least one position. For example, \clusterOne{} is a $3$-cluster of $3122$, while \clusterTwo{} is a $4$-cluster of $3122$ with the same underlying word but one additional marking. 

Define the \emph{cluster generating function} of $u$ by
\[
	C_u(x,y,z) = \sum_{m\ge 1} z^m\; \sum_{\mathclap{\substack{\text{$m$-clusters}\\\text{$c$ of $u$}}}} x^{|c|}y^{\|c\|}.
\]
It follows that the generating function $A_u(x,y,z)$, which counts words by length, sum, and the number of occurrences of $u$, can be derived from $C_u(x,y,z)$ via the formula
\[
	A_u(x,y,z) = \frac{1}{1-\frac{xy}{1-y} - C_u(x,y,z-1)}.
\]
Therefore, $u$ and $v$ are strongly Wilf equivalent if and only if $C_u(x,y,z) = C_v(x,y,z)$ and Wilf equivalent if and only if $C_u(x,y,-1) = C_v(x,y,-1)$. Hence one can prove Wilf equivalence results directly using cluster generating functions.

We make one further simplification. A \emph{minimal $m$-cluster of $u$} is an $m$-cluster of $u$ such that no letter can be decreased without destroying a marked occurrence of $u$. The example \clusterOne{} above is not a minimal $3$-cluster of $3122$, but \clusterThree{} is. Define the \emph{minimal cluster generating function} of $u$ by
\[
	M_u(x,y,z) = \sum_{m\ge 1} z^m\; \sum_{\mathclap{\substack{\text{minimal}\\ \text{$m$-clusters}\\\text{$c$ of $u$}}}} x^{|c|}y^{\|c\|}.
\]
The generating functions for clusters and minimal clusters are related by the equality
\[
	C_u(x,y,z) = M_u\left(\frac{x}{1-y}, y, z\right),
\]
and therefore $u$ and $v$ are strongly Wilf equivalent if and only if $M_u(x,y,z) = M_v(x,y,z)$. It follows that one can prove that two words $u$ and $v$ are strongly Wilf equivalent by exhibiting a bijection between minimal clusters of $u$ and minimal clusters of $v$ that preserves length, sum, and number of marked occurrences: this is the technique used in the main result of the next section.


\section{Rigid Shifts and Strong Wilf equivalence}

The \emph{skyline diagram} of a word $u = u_1u_2 \ldots u_n \in \P^*$ of length $n$ is the geometric figure formed by adjoining $n$ columns of squares such that the $i$th column is made up of $u_i$ squares. For this definition a picture proves more useful: Figure~\ref{figure:skyline-diagram-example} shows the skyline diagrams of $241625$ and $122213132$.

\begin{figure}
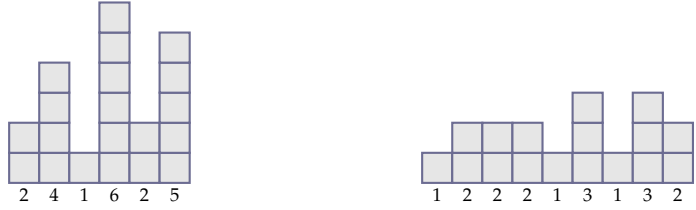

\begin{center}
	\ \hfill
	\sds{2,4,1,6,2,5}{0.4}
	\hfill
	\sds{1,2,2,2,1,3,1,3,2}{0.4}
	\hfill\ 
	
	\caption{On the left, the skyline diagram of $241625$. On the right, the skyline diagram of $122213132$.}
	\label{figure:skyline-diagram-example}
\end{center}
\end{figure}

With this perspective, one can think of a minimal $m$-cluster of $u$ as $m$ overlapped copies of the skyline diagram of $u$ that together create a larger skyline diagram. For example, the minimal $3$-cluster \clusterThree{} of $3122$ is shown in Figure~\ref{figure:skyline-minimal-cluster}.

\begin{figure}
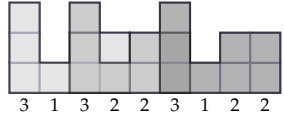

\begin{center}
	\sdsplus{3,1,3,2,2,3,1,2,2}{0.4}{
	
		\draw[thick, fill=gray!20, opacity=0.5] (1,1)--(1,4)--(2,4)--(2,2)--(3,2)--(3,3)--(5,3)--(5,1)--cycle;
		\begin{scope}[shift={(2,0)}]
			\draw[thick, fill=gray!60, opacity=0.5] (1,1)--(1,4)--(2,4)--(2,2)--(3,2)--(3,3)--(5,3)--(5,1)--cycle;		
		\end{scope}
		\begin{scope}[shift={(5,0)}]
			\draw[thick, fill=gray, opacity=0.5] (1,1)--(1,4)--(2,4)--(2,2)--(3,2)--(3,3)--(5,3)--(5,1)--cycle;		
		\end{scope}
	
	}
	\caption[Three overlaid copies of the skyline diagram of $3122$ starting in positions $1$, $3$, and $6$. The resulting skyline diagram is that of the minimal $3$-cluster $313223122$ of $3122$.]{Three overlaid copies of the skyline diagram of $3122$ starting in positions $1$, $3$, and $6$. The resulting skyline diagram is that of the minimal $3$-cluster \clusterThree{} of $3122$.}
	\label{figure:skyline-minimal-cluster}
\end{center}
\end{figure}

A \emph{rigid shift} of a word $u$ is any word $v$ that can be formed by cutting the skyline diagram of $u$ at some height $h$ and rigidly moving together all blocks above the cut line in such a way that each moved column comes to rest on a column of height exactly $h$. To illustrate, consider the word $u = 2233213452$. Figure~\ref{figure:skyline-rigid-shift-1} shows the word $u$ and two rigid shifts of $u$, while Figure~\ref{figure:skyline-rigid-shift-2} demonstrates three deformations of $u$ that are not rigid shifts. 

\begin{figure*}
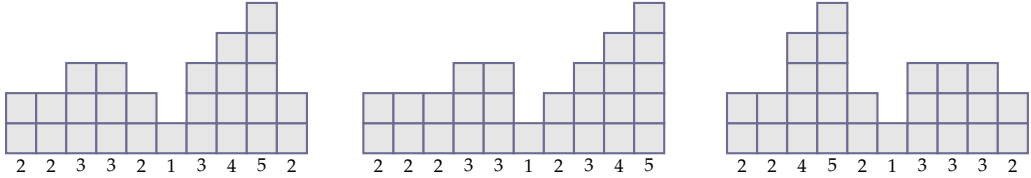

	\centering
	\begin{subfigure}[t]{0.3\textwidth}
		\centering
		\sds{2,2,3,3,2,1,3,4,5,2}{0.4}
	\end{subfigure}%
	~ 
	\begin{subfigure}[t]{0.3\textwidth}
		\centering
		
		\sds{2,2,2,3,3,1,2,3,4,5}{0.4}

	\end{subfigure}
	~ 
	\begin{subfigure}[t]{0.3\textwidth}
		\centering
		
		\sds{2,2,4,5,2,1,3,3,3,2}{0.4}	
		
	\end{subfigure}
	
	\caption{On the left, the skyline diagram of $u=2233213452$. In the middle, a rigid shift performed by cutting the diagram at height $2$ and shifting all blocks above height $2$ one column to the right. On the right, a rigid shift performed by cutting the diagram at height $3$ and shifting all blocks above height $3$ five columns to the left.}
	
	\label{figure:skyline-rigid-shift-1}
\end{figure*}

\begin{figure*}[t]
	\centering
	\begin{subfigure}[t]{0.3\textwidth}
		\centering
		\sds{2,2,3,3,2,3,4,5,2,1}{0.4}
	\end{subfigure}%
	~ 
	\begin{subfigure}[t]{0.3\textwidth}
		\centering
		
		\begin{tikzpicture}[baseline=(current bounding box.south),scale={0.4}]
			\foreach[count=\x] \i in {2,3,3,2,1,3,4,5,2,1}{
				\foreach \j in {1,...,\i}{
					\draw[gray!85!blue, thick, fill=gray!20] (\x, \j) rectangle (\x+1,\j+1);
				}
				\node at (\x+.5,.6) {\scriptsize $\i$};
			}
			\draw[gray!85!blue, thick, fill=gray!20] (0, 2) rectangle (1,3);
			\node at (0.5, 0.6) {\scriptsize ?};
		\end{tikzpicture}

	\end{subfigure}
	~ 
	\begin{subfigure}[t]{0.3\textwidth}
		\centering
		
		\begin{tikzpicture}[baseline=(current bounding box.south),scale={0.4}]
			\foreach[count=\x] \i in {2,3,3,2,2,1,4,5,2,2}{
				\foreach \j in {1,...,\i}{
					\draw[gray!85!blue, thick, fill=gray!20] (\x, \j) rectangle (\x+1,\j+1);
				}
				\node at (\x+.5,.6) {\scriptsize $\i$};
			}
			\draw[gray!85!blue, thick, fill=gray!20] (6, 3) rectangle (7,4);
			\draw[white,fill=white] (6,0) rectangle (7,0.8);
			\node at (6.5, 0.6) {\scriptsize ?};
		\end{tikzpicture}

	\end{subfigure}
	
	\caption{Three deformations of $u=2233213452$ that are not rigid shifts. On the left, only some, but not all, of the blocks of height at least two have been shifted. In the middle, all blocks of height at least two have been shifted to the left by one column, causing an illegal overhang. On the right, all blocks of height at least three have been shifted to the left by one column, creating a unconnected column of blocks.}
	
	\label{figure:skyline-rigid-shift-2}
\end{figure*}
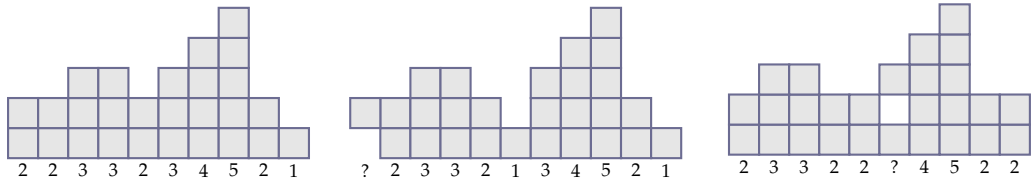

The \emph{shift equivalence class} of a word  $u$ is the set of all words that can be obtained by starting with $u$ and performing any sequence of reversals and rigid shifts. By reversals, we mean reversing the order of the letters in the entire word, not just in some subword. We will show in this section that the shift equivalence relation is a refinement of the strong Wilf equivalence relation; that is, any two shift equivalent words must also be strongly Wilf equivalent. The concreteness and geometric nature of the shift equivalence relation is thus a powerful tool for studying Wilf equivalence in the generalized factor order.

\begin{theorem}
	Any two shift equivalent words are strongly Wilf equivalent.
\end{theorem}
\begin{proof}
	First we prove that a word and its reverse are strongly Wilf equivalent. Although this case is fairly trivial, it serves as a nice warm up for the latter part of the proof. Let $u^R$ denote the reverse of $u$. We will show that $u$ and $u^R$ are strongly Wilf equivalent by exhibiting a bijection between minimal clusters of $u$ and minimal clusters of $u^R$ that preserves length, sum, and the number of marked occurrences. Let $\Phi$ be the map from $m$-clusters of $u$ to $m$-clusters of $u^R$ defined by $\Phi(c) = c^R$, where in $c^R$ the placement of the markings is correspondingly reversed. As this is clearly a bijection with the desired properties, $M_u(x,y,z) = M_{u^R}(x,y,z)$. By the results of Section~\ref{section:cluster}, $u$ and $u^R$ are strongly Wilf equivalent. 
	
	We move now to the harder case. Let $v$ be formed from $u$ by performing a rigid shift. We will show that $v$ and $u$ are strongly Wilf equivalent. For concreteness, let $h$ be the height below which all blocks remain in place and above which all blocks shift, and let $k$ be the horizontal shift performed ($k$ is negative if blocks are shifted to the left, and positive if shifted to the right). Let $c$ be a minimal $m$-cluster of $u$ where the $m$ marked occurrences of $u$ begin at positions $1 = i_1 < i_2 < \ldots < i_m$. Define $\Phi(c)$ be the minimal $m$-cluster of $v$ with marked occurrences of $v$ beginning at the same positions $i_1, \ldots, i_m$. To complete the proof, we must show that $\Phi$ is a bijection from minimal clusters of $u$ to minimal clusters of $v$ that preserves length, sum, and number of marked occurrences.
	
	The bijectivity of $\Phi$ follows immediately after observing that $u$ and $v$ have the same length and that $c$ and $\Phi(c)$ are the unique minimal $m$-clusters of $u$ and $v$, respectively, that have marked occurrences at positions $i_1, \ldots, i_m$. It is similarly easy to see that $\Phi$ preserves length and number of marked occurrences. To establish that $\Phi$ is sum-preserving, it suffices to show that $\Phi(c)$ is itself a rigid shift of $c$. In fact, $\Phi(c)$ is obtained from $c$ by performing the same type of rigid shift performed on $u$ to obtain $v$: all blocks of $c$ above height $h$ are shifted horizontally by $k$ units. Before we verify this fact, we present a graphical example. Suppose $u = 252432122$ and $v = 222225143$. The corresponding skyline diagrams are shown below.
	\begin{center}
		\sds{2,5,2,4,3,2,1,2,2}{0.4}
		\hspace*{0.8in}
		\sds{2,2,2,2,2,5,1,4,3}{0.4}
	\end{center}
	One transforms $u$ into $v$ by rigidly shifting all blocks of height at least $3$ rightward by four units. The minimal $4$-clusters of $u$ and $v$ with marked occurrences starting at positions $1, 3, 6, 10$ are shown below.
	\begin{center}
		\sdsplus{2,5,2,5,3,4,5,2,4,3,5,2,4,3,2,1,2,2}{0.4}{
			\draw[thick, fill=black!70, opacity=0.8] (1,1)--(1,3)--(16,3)--(16,2)--(17,2)--(17,3)--(19,3)--(19,1)--cycle;
					
			\draw[thick, fill=black!30, opacity=0.8] (2,3)--(2,6)--(3,6)--(3,3)--(4,3)--(4,6)--(5,6)--(5,4)--(6,4)--(6,5)--(7,5)--(7,6)--(8,6)--(8,3)--(9,3)--(9,5)--(10,5)--(10,4)--(11,4)--(11,6)--(12,6)--(12,3)--(13,3)--(13,5)--(14,5)--(14,4)--(15,4)--(15,3)--cycle;
		}
		\quad\quad 
		\sdsplus{2,2,2,2,2,5,2,5,3,4,5,2,4,3,5,1,4,3}{0.4}{
			\draw[thick, fill=black!70, opacity=0.8] (1,1)--(1,3)--(16,3)--(16,2)--(17,2)--(17,3)--(19,3)--(19,1)--cycle;
					
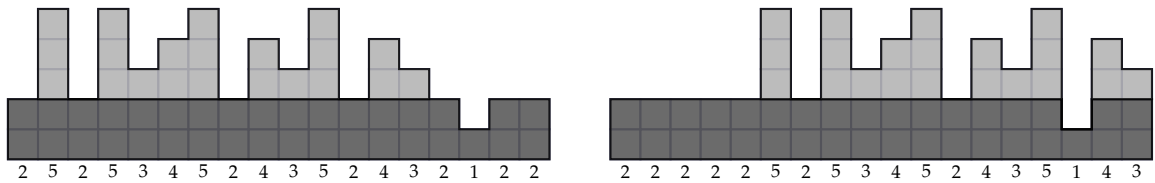
			\begin{scope}[shift={(4,0)}]
				\draw[thick, fill=black!30, opacity=0.8] (2,3)--(2,6)--(3,6)--(3,3)--(4,3)--(4,6)--(5,6)--(5,4)--(6,4)--(6,5)--(7,5)--(7,6)--(8,6)--(8,3)--(9,3)--(9,5)--(10,5)--(10,4)--(11,4)--(11,6)--(12,6)--(12,2)--(13,2)--(13,5)--(14,5)--(14,4)--(15,4)--(15,3)--(13,3)--(13,2)--(12,2)--(12,3)--cycle;				
			\end{scope}
		}
	\end{center}
	The darkly-shaded boxes are those that do not shift, and they are identical between the minimal $4$-cluster of $u$ on the left and the minimal $4$-cluster of $v$ on the right. The lightly-shaded boxes correspond to those that shift to form $v$ from $u$. There is no worry that shifting blocks of the cluster will form an unconnected column (as in the rightmost skyline diagram of Figure~\ref{figure:skyline-rigid-shift-2}) because the presence of such a defect would imply that $v$ itself had the same defect.
	
	To complete the proof, we now verify algebraically that $\Phi(c)$ is obtained from $c$ by performing the same type of rigid shift performed on $u$ to obtain $v$. Since $v$ is formed from $u$ by shifting all blocks of height greater than $h$ horizontally by $k$ units, we can describe each letter of $v$ in terms of the letters of $u$ in following way:
	\[
		v_n = \min(h, u_n) + \max(0, u_{n-k}-h),
	\]	
	with the convention that $u_n=0$ if $n < 0$ or $n > |u|$. The following property follows immediately from the definition of a rigid shift.
	\begin{lemma}
		\label{lemma:shift-prop}
		If $u_n < h$, then $u_{n-k} < h$. Equivalently, if $\min(h, u_n) < h$, then $\max(0, u_{n-k}-h) = 0$.
	\end{lemma}

	Furthermore, the letters of $c$ and $\Phi(c)$ can be written as
	\[
		c_n = \max_{j=1,\ldots,m}(u_{n-i_j+1}),\qquad\text{and}\qquad\Phi(c)_n = \max_{j=1,\ldots,m}(v_{n-i_j+1}).
	\]
	
	We are ready to prove that $\Phi(c)$ is the rigid shift of $c$ at height $h$ by $k$ units.
	\begin{align*}
		\Phi(c)_n &= \max_{j=1,\ldots,m}(v_{n-i_j+1})\\	
		&= \max_{j=1,\ldots,m} \left( \min(h, u_{n-i_j+1}) + \max(0, u_{n-i_j+1-k}-h) \right)\\
		&= \max_{j=1,\ldots,m} \left( \min(h, u_{n-i_j+1}) \right) + \max_{j=1,\ldots,m}\left(\max(0, u_{n-i_j+1-k}-h)\right)\\
		&= \min\left(h, \max_{j=1,\ldots,m}(u_{n-i_j+1})\right) + \max\left(0, \max_{j=1,\ldots,m} (u_{n-i_j+1-k})-h\right)\\
		&= \min(h,c_n) + \max(0, c_{n-k}-h).
	\end{align*}
	
	The third line follows from Lemma~\ref{lemma:shift-prop}, and the last line shows that $\Phi(c)$ is the claimed rigid shift of $c$. As a consequence, $\Phi$ is sum-preserving.\end{proof}

\section{Applications of Shift Equivalence}

The notion of shift equivalence provides a uniform framework that can be used to prove many of the results of Kitaev, Liese, Remmel, and Sagan~\cite{KLRS} and of Pantone and Vatter~\cite{PV}. 

Kitaev, Liese, Remmel, and Sagan~\cite{KLRS} specifically ask for such a framework, stating the following as an open question about Wilf equivalence over words in $S_5$, i.e., those of length $5$ that contain each letter of $\{1,2,3,4,5\}$ exactly once:

\begin{quote}
	\emph{Find a theorem which, together with the results already proved, explains all the Wilf equivalences in $S_5$.}
\end{quote}

They state the following as a conjecture, subsequently proved by Pantone and Vatter~\cite{PV} by an analysis of 15 cases.

\begin{theorem}[{\cite[Conjecture 8.3]{KLRS} and \cite[Theorem 5.3]{PV}}]
	For any $a,b,c \in \P\setminus\{1\}$, the words $a1b2c$	and $a2b1c$ are Wilf equivalent.
\end{theorem}

One can easily see that $a1b2c$ and $a2b1c$ are shift equivalent by considering the possible relative sizes of $a$, $b$, and $c$. Therefore, they are strongly Wilf equivalent and the result follows. In fact, as noted by Pantone and Vatter~\cite{PV}, the theorem still holds by the same proof if $1$ and $2$ are replaced by arbitrary positive integers $x$ and $y$ such that $x,y \leq a,b,c$.

The next theorem was originally proved by Kitaev, Liese, Remmel, and Sagan~\cite{KLRS}, the proof requiring about a page. It follows immediately from the concept of shift equivalence.

\begin{theorem}[{\cite[Theorem 4.3]{KLRS}}]
	Let $u$, $v$, and $w$ be words over $\{1, \ldots, m\}$ and let $n \geq m$. Then, the words $umvnw$ and $unvmw$ are Wilf equivalent.
\end{theorem}

The words $unvmw$ can be created from the word $umvnw$ by shifting the $n-m$ topmost blocks of the column of height $n$ to the column of height $m$. Therefore, the words are strongly Wilf equivalent.

Other results follow with slight modifications, such as replacing the assumption ``if $u$ and $v$ are Wilf equivalent'' to the assumption ``if $u$ and $v$ are shift equivalent.'' We present just one example here. A word $u$ is said to be \emph{weakly increasing} (resp., \emph{weakly decreasing}) if $u_i \geq u_{i-1}$ (resp., $u_i \leq u_{i-1}$) for $2 \leq i \leq |u|$. 

\begin{theorem}[{Analogue of~\cite[Corollary 4.2]{KLRS}}]
	Let $y$ and $y'$ be weakly increasing words and $z$ and $z'$ be weakly decreasing words such that $yz$ is a rearrangement of $y'z'$. Let $m$ be the maximum letter in $yz$. If $u$ and $v$ are shift equivalent words with no letter less than $m$, then $yuz$ and $y'vz'$ are strongly Wilf equivalent.
\end{theorem}

\section{Open Questions}

Shift equivalence and strong Wilf equivalence are not the same relation. Using the computational techniques of Kitaev, Liese, Remmel, and Sagan~\cite{KLRS}, appropriately adapted to calculate the trivariate generating function $A_u(x,y,z)$ rather than the bivariate $A_u(x,y,0)$, we can find examples of pairs of words that are strongly Wilf equivalent but not shift equivalent.

Examples of least length are found among the permutations of $\{1, 2, 3, 4, 5, 6\}$. Here there are three strong Wilf equivalence classes that each split into two shift equivalence classes: each of the pairs $(234156, 256143)$, $(235146, 254163)$, and $(235164, 245163)$ consists of words that are strongly Wilf equivalent but not shift equivalent. The unique pair of strongly Wilf equivalent but not shift equivalent words whose digits sum to $14$ is $(223133, 233132)$, and there is no such pair with a smaller sum. 

There may be more geometric relations, like shift equivalence, that imply strong Wilf equivalence and can explain the examples above. Ideally, one would like a small set of such geometric relations that together identify all strongly Wilf equivalent pairs. 

One may also ask how many strong Wilf equivalence classes there are among the permutations $\{1, \ldots, n\}$. Computation shows that this sequence starts $\{1, 1, 2, 5, 21, 126, 922\}$, while the counting sequence for the number of shift equivalence classes of $\{1, \ldots, n\}$ starts $\{1, 1, 2, 5, 21, 129, 931\}$. The highly structured nature of shift equivalence could lead to exact enumeration of the latter sequence, providing an upper bound for the former.
%
%
%
%
%
%
%

\end{document}